\documentclass[11pt]{amsart}
\usepackage{amsmath,amssymb}

\begin{document}

\newcommand{\chp}{\mathds{1}}

\newtheorem{thm}{Theorem}[section]
\newtheorem{theorem}{Theorem}[section]
\newtheorem{lem}[thm]{Lemma}
\newtheorem{lemma}[thm]{Lemma}
\newtheorem{prop}[thm]{Proposition}
\newtheorem{proposition}[thm]{Proposition}
\newtheorem{cor}[thm]{Corollary}
\newtheorem{conj}[thm]{Conjecture}

\theoremstyle{definition}
\newtheorem{defn}[thm]{Definition}
\newtheorem*{remark}{Remark}
\newtheorem{objective}[thm]{Objective}


\newcommand{\Z}{{\mathbb Z}} 
\newcommand{\Q}{{\mathbb Q}}
\newcommand{\R}{{\mathbb R}}
\newcommand{\C}{{\mathbb C}}
\newcommand{\N}{{\mathbb N}}
\newcommand{\FF}{{\mathbb F}}
\newcommand{\fq}{\mathbb{F}_q}
\newcommand{\rmk}[1]{\footnote{{\bf Comment:} #1}}

\newcommand{\Gal}{\mathrm{Gal}}
\newcommand{\Fr}{\mathrm{Fr}}
\newcommand{\Hom}{\mathrm{Hom}}
\newcommand{\GL}{\mathrm{GL}}

\renewcommand{\mod}{\;\operatorname{mod}}
\newcommand{\ord}{\operatorname{ord}}
\newcommand{\TT}{\mathbb{T}}
\renewcommand{\i}{{\mathrm{i}}}
\renewcommand{\d}{{\mathrm{d}}}
\renewcommand{\^}{\widehat}
\newcommand{\HH}{\mathbb H}
\newcommand{\Vol}{\operatorname{vol}}
\newcommand{\area}{\operatorname{area}}
\newcommand{\tr}{\operatorname{tr}}
\newcommand{\norm}{\mathcal N} 
\newcommand{\intinf}{\int_{-\infty}^\infty}
\newcommand{\ave}[1]{\left\langle#1\right\rangle} 
\newcommand{\Var}{\operatorname{Var}}
\newcommand{\Prob}{\operatorname{Prob}}
\newcommand{\sym}{\operatorname{Sym}}
\newcommand{\disc}{\operatorname{disc}}
\newcommand{\CA}{{\mathcal C}_A}
\newcommand{\cond}{\operatorname{cond}} 
\newcommand{\lcm}{\operatorname{lcm}}
\newcommand{\Kl}{\operatorname{Kl}} 
\newcommand{\leg}[2]{\left( \frac{#1}{#2} \right)}  
\newcommand{\Li}{\operatorname{Li}}

\newcommand{\sumstar}{\sideset \and^{*} \to \sum}
\newcommand{\prodstar}{{\sideset \and^{*} \to \prod}}

\newcommand{\LL}{\mathcal L} 
\newcommand{\sumf}{\sum^\flat}
\newcommand{\Hgev}{\mathcal H_{2g+2,q}}
\newcommand{\USp}{\operatorname{USp}}
\newcommand{\conv}{*}
\newcommand{\dist} {\operatorname{dist}}
\newcommand{\CF}{c_0} 
\newcommand{\kerp}{\mathcal K}

\newcommand{\Cov}{\operatorname{cov}}
\newcommand{\Sym}{\operatorname{Sym}}

\newcommand{\ES}{\mathcal S} 
\newcommand{\EN}{\mathcal N} 
\newcommand{\EM}{\mathcal M} 
\newcommand{\Sc}{\operatorname{Sc}} 
\newcommand{\Ht}{\operatorname{Ht}}

\newcommand{\E}{\operatorname{E}} 
\newcommand{\sign}{\operatorname{sign}} 

\newcommand{\divid}{d} 

\newcommand{\pol}{\mathcal P_d}
\newcommand{\Bad}{\operatorname{Bad}}

\newcommand{\pols}{\mathcal P_\gamma(T)}

\title[On Cilleruelo's conjecture]{On Cilleruelo's conjecture for the least common multiple of polynomial sequences}
\author{Ze\'ev Rudnick and Sa'ar Zehavi}

 \address{Raymond and Beverly Sackler School of Mathematical Sciences,
Tel Aviv University, Tel Aviv 69978, Israel}
 
 \begin{abstract}
 A conjecture due to Cilleruelo states that  for an irreducible polynomial $f$ with integer coefficients of degree $d\geq 2$, the least common multiple $L_f(N)$ of the sequence $f(1), f(2), \dots, f(N)$  has asymptotic growth $\log L_f(N)\sim (d-1)N\log N$ as $N\to \infty$. We establish a version of this conjecture for almost all shifts of a fixed polynomial, the range of $N$ depending on the range of shifts. \end{abstract}
\date{\today}
\maketitle

 \section{Introduction}
 
 \subsection{Background}
 It is a well known and elementary fact that the least common multiple of all integers $1,2,\dots,N$ is exactly given by 
$$\log \lcm\{1,2,\dots,N\} = \psi(N):=\sum_{n\leq N} \Lambda(n)$$
with $\Lambda(n)$ being the von Mangoldt function, and hence by the Prime Number Theorem, 
$$
\log \lcm\{1,2,\dots,N\} \sim N .
$$

For a polynomial $f\in \Z[X]$, set 
$$L_f(N):=\lcm\{f(n):n=1,\dots,n\}.
$$
The goal is to understand the asymptotic growth of $\log L_f(N)$ as $N\to \infty$. 

In the linear case $\deg f=1$, we still have $\log L_f(N)\sim c_f N $ from the Prime Number Theorem in arithmetic progressions, see 
 e.g. \cite{BKS}. A similar growth occurs for products of linear polynomials, see \cite{HQT}, and for any polynomial with {\em non-negative} integer coefficients, there is a lower bound $\log L_f(N)\gg N$ \cite{HLQW}. 
  However, in the case of irreducible polynomials higher degree, 
Cilleruelo \cite{Cilleruelo} 
conjectured that the growth is faster than linear, precisely: 
\begin{conj}\label{conj cil} 
If $f$ is an irreducible polynomial with $\deg f\geq 2$, then
$$\log L_f(N)\sim (\deg f-1)N\log N,\quad N\to \infty .
$$
\end{conj}
 Cilleruelo proved Conjecture~\ref{conj cil} for quadratic polynomials. Moreover,  in that case there is a secondary main  term 
\begin{equation*}\label{refined cil}
\log L_f(N)=N\log N + b_fN +o(N)
\end{equation*}
see also \cite{RSZ}. 
No other case of Conjecture~\ref{conj cil} is known to date.   
We do know that for any irreducible $f$ of degree $d\geq 3$, we have an upper bound
$\log L_f(N)\lesssim (d -1) N\log N$    
and one can prove $\log L_f(N)\gg N\log N$,  
as came up in a discussion with James Maynard \cite{MaynardRudnick}.

We will show that Conjecture~\ref{conj cil} holds  {\em for almost all}  $f$ in a suitable sense.

\subsection{General setup} 

We fix a polynomial  $f_0(x)\in \Z[x]$ of degree $d\geq 3$, which we assume is monic (or more generally, primitive - no prime divides all coefficients) and for $a\in \Z$ we set
\[
f_a(x) = f_0(x)-a
\]
 It is known   that  they are generically irreducible.  Set 
\[
L_a(N) = \lcm\{ f_0(n)-a:n=1,\dots, N\} 
\]
We want to show that 
\begin{thm}\label{thm f0(x)-a}
For almost all $|a|\leq T$,  and all $N$ satisfying  
\[
 T^{\frac{1}{d-1}}< N< \frac{T}{\log T}
\]
we have 
\begin{equation}\label{result on La(N)}
\log L_a(N)\sim (d-1)N\log N.
\end{equation}
\end{thm}
\begin{remark}
What one would like to show is that \eqref{result on La(N)} holds for \underline{all} $N>N_0(a)$, for all but $o(T)$ values of $|a|\leq T$. At this time we do not know how to do this. 
\end{remark}
\subsection{Plan}
Let 
\[
P_a(N) = \prod_{n\leq N}|f_0(n)-a| 
\] 
We write down the prime power factorization
\[
P_a(N)= \prod_p p^{\alpha_p(a;N)}
\]
 If $T\ll N^{d-1}$ then $\alpha_p(a;N)=0$ for $p\gg N^d$, and (Lemma~\ref{lem:asymp of P}) 
\[
\log P_a(N) = d N\log N +O(N).
\]

We also write the prime power factorization of $L_a(N)$ as
\[
L_a(N)= \prod_p p^{\beta_p(N)}
\]
Let 
\[
D(a)  = \disc(f_0(x)-a)
\]
be the discriminant of $f_0(x)-a$. It is a polynomial in $a$ of degree $d-1$, with integer coefficients. 

We show (Proposition~\ref{prop: decomp})
\begin{equation}\label{basic formula for L}
\log L_a(N) = dN\log N - \Bad_N(a) - \Delta_N(a) -NC_N(a) +O(N)
\end{equation}
where
\[
\Bad_N(a) =\sum_{\substack{p\leq N\\ p\mid D(a)}} \alpha_p(N)\log p
\]
\[
\Delta_N(a) = \sum_{N<p\ll N^d} (\alpha_p(N)-\beta_p(N)) \log p
\]
\[
C_N(a) =\sum_{\substack{  p\leq N\\p\nmid D(a)}}\frac{\log p}{p-1} \rho(a;p)
\]
where $\rho(a;d) =   \#\{n\bmod d: f_0(n)-a=0\bmod d\}$.

We will show that for almost all $|a|\leq T$, with $N\log N<T<N^{d-1}$, we have 
\begin{equation} \label{result CN} 
C_N(a)\sim \log N  .
\end{equation}
\begin{equation} \label{result Bad} 
\Bad_N(a)\ll N (\log\log N)^{1+o(1)}.
\end{equation}
\begin{equation}\label{result Delta}
\Delta_N(a)\ll N (\log\log N)^{1+o(1)}.
\end{equation} 
Inserting these into \eqref{basic formula for L} will prove Theorem~\ref{thm f0(x)-a}. 


To prove \eqref{result CN},  \eqref{result Bad}, \eqref{result Delta} we use averaging: Denoting by $\ave{\bullet}$ 
the average over all $|a|\leq T$ such that $f_0(x)-a$ is irreducible, we show that for $N\log N<T<N^{d-1}$,
\begin{equation} \label{variance CN} 
\ave{|C_N(a)-\log N|^2}\ll (\log \log N)^2 
\end{equation}
\begin{equation} \label{ave Bad} 
\ave{\Bad_N(a)}\ll N \log\log N
\end{equation}
\begin{equation}\label{ave Delta}
\ave{\Delta_N(a)}\ll N \log\log N.
\end{equation} 
Noting that $\Delta_N(a),\Bad_N(a)\geq0$ are non-negative, we obtain \eqref{result CN},  \eqref{result Bad}, \eqref{result Delta} 
from the Chebysehv/Markov inequality.

\begin{remark}
In the deterministic case ($a$ fixed, $N\to \infty$), the quantities $\Bad_N$ and $C_N$ can be handled easily, as in the quadratic case $d=2$, see \cite{Cilleruelo}. It is the quantity $\Delta_N(a)$ which we do not know how to  show is $o(N\log N)$ individually (though the upper bound $O(N\log N)$ is easy). This is why we need to average over $a$. However, letting $a$ grow with $N$ introduces new problems, in particular for the study of $C_N$, which may need GRH to overcome individually. The results \eqref{variance CN}  and  \eqref{ave Bad}  for random $a$ are much easier and this is the method that we use.  
\end{remark}

\noindent{\bf Acknowledgements.} 
 We thank Shaofang Hong and Guoyou Qian for introducing the problem during a visit to Chengdu in 2017, and Lior Bary Soroker and James Maynard 
 for discussions. 
The research   was supported by the European Research Council under the European Union's Seventh
Framework Programme (FP7/2007-2013) / ERC grant agreements n$^{\text{o}}$ 320755 and 786758.

 \section{Background}
 \subsection{Generic irreducibility} 
Fix $f_0(x)\in \Z[x]$ monic, of degree $d\geq 2$. 
It is known that $f_0(x)-a$ is generically irreducible,  in fact (see \cite[\S9.7]{Serre})  
  \begin{lem}\label{lem:Serre} 
Fix $f_0(x)\in \Z[x]$ of degree $d\geq 2$. Then the number of $|a|\leq T$ for which $f_0(x)-a$ is reducible is $O(\sqrt{T})$.
  \end{lem}
  This is sharp in this generality, since for even degree $d=2m$, for the polynomial $f_0(x) = x^{2m}$ we have $x^{2m}-a$ is reducible whenever $a=b^2$ is a perfect square. 
  
    Denote
\[
D(a)=\disc(f_0(x)-a)
\]
the discriminant of $f_a(x)$, 
which is a polynomial in $a$ of degree $\leq d-1$ with integer coefficients (depending on the coefficients of $f_0$). 
We assume that $a$ is such that  $f_0(x)-a$ is irreducible, and therefore $D(a)$  is not  zero, i.e. $f_a$ has no multiple roots.

Examples: 

i) $f_0(x) = x^3$, then $\disc(f_0(x)-a) = \disc(x^3-a) = -27a^2$. 

ii) $f_0(x)=x^3-3x$ then $\disc(f_0(x)-a) =-27(a-2)(a+2)$.

 

\subsection{A decomposition}

\begin{prop}\label{prop: decomp}
 For $|a|\leq N^{d-1}$ such that $f_0(x)-a$ is irreducible, we have
 \[
 \log L_a(N) = d\log N  -\Bad_N(a)  -NC_N(a) 
-\Delta_N(a)  +O(N) 
 \]
 \end{prop}
 
 For  $a\in \Z$ such that  $f_a(x)=f_0(x)-a$ is irreducible, let 
$$P_a(N):=\prod_{n\leq N} |f_a(n)|
$$
 which is nonzero since $f_a$ has no rational roots, and write the prime power decomposition as 
$$
P_a(N) = \prod_p p^{\alpha_p(N)}  
$$
so that 
\[
\alpha_p(N) = \sum_{n\leq N} \nu_p(f_a(n))
\]
where $\nu_p(m):=\max(k\geq 0:p^k\mid m)$.

Following Cilleruelo \cite{Cilleruelo}, we want to relate $\log L_a(N)$ to $\log P_a(N)$, which is clearly bigger. 
We  write the prime power decomposition of $L_a(N)$ as 
$$L_a(N) = \prod_p p^{\beta_p(N)}$$
 with 
$$ \beta_p(N) = \max\{\nu_p(f_a(n)):n\leq N\} .
$$

Using the prime factorization of $L_a(N)$ and $P_a(N)$ we have 
\begin{multline}\label{decomp prelim}
\log L_a(N) = \log P_a(N) - \sum_{p\leq N }\alpha_p(N)\log p+\sum_{p\leq N }\beta_p(N)\log p 
\\
- \sum_{p> N } (\alpha_p(N)-\beta_p(N))\log p
\end{multline}
where we have separated out the contribution of primes $p\leq N $,   and the larger ones.  We further break off the contribution of primes $p\leq N $ which divide the discriminant $D(a)=\disc(f_a)$, by setting
\[
\Bad_N(a):=\sum_{\substack{p\leq N  \\ p\mid D(a)}}\alpha_p(N)\log p
\]
and abbreviate the contribution of big primes $p>N $ as 
\[
\Delta_N(a):=\sum_{p> N } (\alpha_p(N)-\beta_p(N))\log p
\]
Note that  $\Bad_N, \Delta_N\geq 0$ are both non-negative. We obtain an expression
\begin{equation}\label{initial division}
\begin{split}
 \log L_a(N) &= \log P_a(N) +\sum_{p\leq N }\beta_p(N)\log p 
\\
&-\Bad_N(a) -  \sum_{\substack{p\leq N \\p\nmid D(a)}}\alpha_p(N)\log p -\Delta_N(a) . 
 \end{split}
\end{equation}

\subsection{The quantity $P_a(N)$}

 \begin{lem}\label{lem:asymp of P}
 For    $f_0(x)\in \Z[x]$ monic of degree $d$, and  $|a|\ll N^{d-1}$ so that $f_0(x)-a$ is irreducible, we have
\[
\log P_a(N) =d N\log N  +O(N)
\]
\end{lem}

\begin{proof}
Write
\[
\log P_a(N) = \sum_{n\leq N}\log |f_0(n)-a|. 
\]
Since we assume $f_0(x)-a$ is irreducible, non of the factors $f_0(n)-a$ can vanish so that $\log P_a(N)$ is well defined. 
If $f_0(x) = x^d+c_{d-1}x^{d-1}+\dots$, we have for $n\leq N$
\[
  f_0(n)-a= n^d\Big(1+\frac {c_{d-1}}n+\frac{c_{d-2}}{n^2} + \dots -\frac{a}{n^d}\Big)
\]

Consider first $n$'s satisfying $N/\log N<n\leq N$, for which we use  (recall $|a|\leq N^{d-1}$) 
\[
\log \Big(1+\frac {c_{d-1}}n+\frac{c_{d-2}}{n^2} + \dots -\frac{a}{n^d}\Big)=  O(\frac{(\log N)^d}{N})
\]
so that 
\[
\sum_{\frac{N}{\log N}<n\leq N}\log |f_0(n)-a| = \sum_{\frac{N}{\log N}<n\leq N}\Big(d\log n + O(\frac{(\log N)^d}{N}) \Big) = dN\log N +O(N)
\]

For $1\leq n\leq N/\log N$,  we just use $1\leq |f_0(n)-a|\ll N^{d}$ so that 
$0\leq \log|f_0(n)-a| \ll \log N$, and 
\[
\sum_{n\leq \frac{N}{\log N}} \log |f_0(n)-a|  \ll \sum_{n\leq  N/\log N} \log N \ll N 
\]

Hence 
\[
\log P_a(N) = dN\log N +O(N)
\]
as claimed. 
\end{proof}

\subsubsection{Dealing with $\beta_p(N)$}
For $a$ such that $f_0(x)-a$ is irreducible, 
we have
$$\beta_p(N)\ll  \frac{\log N}{\log p}$$
because 
\[
\beta_p(N) = \max_{n\leq N}\max(k\geq 0:p^k\mid f_0(n)-a)
\]
and since $f_0(n)-a\neq 0$ for all $n$, if $p^k\mid f_0(n)-a\neq 0$ then 
\[
k\leq \frac{ \log|f_0(n)-a|}{\log p}\ll \frac{ \log n+\log|a|}{\log p }
\]
Hence since $|a|\lll N^{d-1}$, 
\[
\beta_p(N) \ll  \frac{\log N}{\log p}
\]
and hence the contribution of primes $p\leq N $ to \eqref{initial division}  is 
\begin{equation}\label{cont of beta}
\sum_{p\leq N }\beta_p(N)\log p\ll \sum_{p\leq N } \log N   \ll N .
 \end{equation}
 
 \subsubsection{Dealing with $\alpha_p(N)$}
 Using Hensel's lemma, it is easy to see that   (\cite{Nagell 1921} see also \cite[Lemma 4] {Cilleruelo}):
 \begin{lem} 
 For $p\nmid D(a)=\disc(f_0(x)-a)$, we have
 \[
 \alpha_p(N) = N\frac{\rho(a;p)}{p-1} +O\Big( \frac{\log N}{\log p} \Big)
 \]
 where $\rho(a;p) = \#\{n\bmod p: f_0(n)-a=0\bmod p\}$. 
  \end{lem}
  Consequently, we find that in \eqref{initial division},
  \[
  \sum_{\substack{p\leq N \\p\nmid D(a)}}\alpha_p(N)\log p  = NC_N(a) + O(N)
  \]
  where
  \[
  C_N(a):=  \sum_{\substack{ p\leq N \\ p\nmid D(a) }}\frac{\log p}{p-1 }\rho (a;p) 
 \]
 Therefore we have proven Proposition~\ref{prop: decomp}.

\section{Bounding $\Bad_N$ almost surely}
Recall that we defined
\[
\Bad_N(a) = \sum_{\substack{p\leq N\\p\mid D(a)}} \log p \sum_{n\leq N} \#\{k\geq 1: p^k\mid f_0(n)-a\} 
\]
(we assume that $f_0(x)-a$ is irreducible).

We denote the averaging operator over $|a|\leq T$ such that $f_0(x)-a$ is irreducible by   
\[
\ave{\bullet} =\frac 1{\#\{  |a|\leq T:  f_0(x)-a \; {\rm is \; irreducible}\}} \sum_{\substack{|a|\leq T\\ f_0(x)-a \;{\rm irreducible} }}\bullet 
\]

The number of $|a|\leq T$ for which $f_0(x)-a$ is reducible is $O(\sqrt{T})$ (Lemma~\ref{lem:Serre}), so that 
\begin{equation}\label{def of ave}
\ave{\bullet} = \frac 1{2T+O(\sqrt{T})}\sum_{\substack{|a|\leq T\\ f_0(x)-a \;{\rm irreducible} }}\bullet 
\end{equation}

\begin{prop}\label{prop:Bad in general}
If $T\geq N$  but $\log T \ll \log N$ then 
\[
\ave{\Bad_N} \ll N\log\log N
\]
\end{prop}
\begin{proof}
We separate out the contribution $B_1(a)$ of $k=1$ and the contribution $B_2(a)$ of the  remaining $k\geq 2$:
\[
\Bad_N(A) = B_1(a)+B_2(a)
\] 
where
\[
B_1(a) = \sum_{\substack{p\leq N\\p\mid D(a)}} \log p \#\{n\leq N: f_0(n)=a\bmod p\}
\]
and
\[
B_2(a) = \sum_{\substack{p\leq N\\p\mid D(a)}} \log p \sum_{n\leq N} \#\{k\geq 2: p^k\mid f_0(n)-a\}
\]
We will show that 
\begin{equation}\label{individ bound for B1} 
B_1(a)\ll N\log\log N
\end{equation} and that 
\[
\ave{B_2}\ll N,
\]
 proving Proposition~\ref{prop:Bad in general}

We first show that 
\[
B_1(a) \ll N\log\log |D(a)|
\]
which suffices for \eqref{individ bound for B1} since $\log |D(a)|\ll \log T \ll \log N$. 

Indeed, for $p\leq N$ we have
\begin{equation*}
\begin{split}
\#\{n\leq N: f_0(n)=a\bmod p\} &=(\frac Np+O(1) ) \#\{n\bmod p: f_0(n)=a\bmod p\} 
\\
&\ll \frac Np \rho(a;p)
\end{split}
\end{equation*}
where
 \[
\rho(a;p):= \#\{n\bmod p: f_0(n)=a\bmod p\} 
\]
which we see by dividing the interval $[1,N]$ into consecutive intervals of length $p$.

Since $f_0(x)$ is a monic polynomial of degree $d$, it is nonzero modulo $p$ and still of degree $\leq d$, hence $\rho(a;p)\leq d$. 
Thus 
\[
B_1(a) \ll \sum_{\substack{p\leq N\\p\mid D(a)}} \log p \; \frac Np     \rho(a;p)
\ll N\sum_{p\mid D(a)} \frac{\log p}{p}   
\]
We use:\footnote{This is standard.} 
\begin{lem}\label{lem divisor sum}
For $k>1$ 
\begin{equation*}
\sum_{p\mid k} \frac{\log p}{p}\ll \log \log k
\end{equation*}
\end{lem}
\begin{proof}
Indeed, splitting the sum into small primes $p\leq \log k$, and the rest (where the summands are at most $\log \log k/\log k$),   we get
\begin{equation*}
\begin{split}
\sum_{p\mid k} \frac{\log p}{p}& 
\leq  \sum_{\substack{p\mid k\\ p\leq \log k}} \frac{\log p}{p} 
+\sum_{\substack{p\mid k\\p>\log k}} \frac{\log p}{p}
\\
& \ll \sum_{p\leq \log k} \frac{\log p}{p}+ \frac{\log \log k}{\log k }\sum_{p\mid k} 1
\\
&\ll \log\log k +  \frac{\log \log k}{\log k }  \cdot\frac{\log k }{\log \log k}
\\
&\ll \log\log k 
\end{split}
\end{equation*}
since the number of  distinct prime divisors of $k$ is $\ll \log k/\log\log k$. 
\end{proof}

Therefore
\[
\sum_{p\mid D(a)} \frac{\log p}{p} \ll \log\log |D(a)| \ll \log \log |a|
\]
and obtain 
\[
B_1(a)\ll N\log\log |D(a)|
\]

Next we bound the mean value of $B_2(a)$
\[
\ave{B_2} = \frac 1{2T+O(\sqrt{T})} \sum_{\substack{|a|\leq T\\ f_0(x)-a\\   {\rm irreducible}  }} \sum_{\substack{p\leq N\\ p\mid D(a)}} \log p \sum_{k\geq 2} \mathbf{1}(f_0(n)=a \bmod p^k )
\]
Now if $f_0(x)-a$ is irreducible, then $f_0(n)-a\neq 0$ and so  if $p^k\mid f_0(n)-a$ with $n\leq N$ then $k\ll \log N/\log p$, 
so we restrict the summation to $2\leq k\ll \log N/\log p$. Moreover, given $n$, the condition $f_0(n)=a\bmod p^k$ determines $a$ modulo $p^k$, so there are $\ll T/p^k+1$ choices for $a$.  
Hence we  may bound
\begin{equation*}
\begin{split}
\ave{B_2} &\ll \frac 1T \sum_{p\leq N} \log p \sum_{n\leq N} \sum_{2\leq k\ll \log N/\log p} (\frac T{p^k}+1) 
\\
&=
\frac NT \sum_{p\leq N} \log p  \sum_{2\leq k\ll \log N/\log p} (\frac T{p^k}+1) 
\end{split}
\end{equation*}
we have
\begin{multline*}
\frac NT \sum_{p\leq N}\log p   \sum_{2\leq k\ll \log N/\log p} \frac T{p^k}\ll N \sum_{p\leq N}  \log p  \sum_{k\geq 2} \frac 1{p^k} 
\ll N\sum_{p\leq N} \frac{\log p}{p^2} \ll N
\end{multline*}
 and
 \[
 \frac NT \sum_{p\leq N} \log p \sum_{2\leq k\ll \log N/\log p}1 \ll  \frac{N}{T} \sum_{p\leq N} \log p \cdot \frac{\log N}{\log p} \ll \frac{N^2}{T}
 \]
 Altogether we find
 \[
 \ave{B_2} \ll N+\frac{N^2}{T}
 \]
 which is $O(N)$ if $T\geq N$. 
 \end{proof}

 \section{Averaging $\Delta_N(a)$}

Let 
\[
\Delta_N(a) = \sum_{p>N} \log p \Big(\alpha_p(N) -\beta_p(N) \Big)
\]
Then clearly $\Delta_N\geq 0$, and we want to show
\begin{prop}\label{prop: mean of DeltaN}
Assume that $T\geq N\log N$, but $\log T\ll \log N$. Then  
\[ 
\ave{\Delta_N} \ll_{f_0} N\log\log N
\]
\end{prop}


\subsection{Preparations}

Let
\[
G(m,n) =\frac{f_0(m)-f_0(n)}{m-n}
\]
which, given $n$, is a (nonzero) polynomial in $m$, of degree $\leq d-1$. If $f_0$ is monic then so is $G(m,n)$ so its degree is exactly $d-1$. 

\begin{lem}\label{lem:G(m,n)} 
There is some $C_1=C_1(f_0)$ so that if  $m,n\geq 1$ and $\max(m,n)>C_1$ then $G(m,n)\neq 0$.
\end{lem}
\begin{proof}
We have 
\[
G(m,n)= \sum_{j=1}^d c_j \frac{m^j-n^j}{m-n}  
\]
and if $j\geq 2$ then for $n=\max(m,n)$, 
\[
\frac{m^j-n^j}{m-n} = n^{j-1}+n^{j-2}m+\dots +m^{j-1} \leq jn^{j-1}
\]
while 
\[
\frac{m^d-n^d}{m-n} = n^{d-1}+n^{d-2}m+\dots +m^{d-1} >n^{d-1} 
\]
so that  (assuming $f_0$ monic, so $c_d=1$)
\[
G(m,n)\geq  \frac{m^d-n^d}{m-n}-\sum_{j=2}^{d-1} |c_j| \frac{m^d-n^d}{m-n} -|c_1| >n^{d-1}-\sum_{j=1}^{d-1} |c_j| j n^{j-1}
\]
which is clearly positive once $n$ is sufficiently large in terms of the coefficients $c_1,\dots ,c_{d-1}$ of $f_0$. 
\end{proof}

\begin{lem}\label{lem: bound on alpha} 
There is some $C(d)>0$ so that for all $|a|\leq N^d$, such that  $f_a(x)=f_0(x)-a$ is irreducible, we have $\alpha_p(N)\leq C(d)$ if  $p>N$.  Moreover $\alpha_p(N)=0$ unless $p\ll N^d+|a|$. 
\end{lem}
\begin{proof}
We have by definition 
\[
\alpha_p(N) = \sum_{n\leq N}\sum_{k\geq 1} \mathbf{1}(f_0(n)=a\bmod p^k) = \sum_{k\geq 1}\#\{n\leq N: f_0(n)=a\bmod p^k\}
\]

Since we assume that $f_a(x)=f_0(x)-a$ is irreducible, hence has no rational zeros, we must have, if $p\mid f_a(n)$, that $p\leq |f_a(n)|\ll N^d+|a|\ll N^d$ uniformly in $|a|\leq T$ (recall $T\leq N^d$). Hence $\alpha_p(N)=0$ for $p\gg N^d$. 

Given $n$ so that $p\mid f_a(n)$, with $p>N$, we claim that there are at most $d$ such integers: 
\[
\#\{ m\leq N: f_a(m)=f_a(n)\bmod p\} \leq d 
\]
Indeed, for any $c\in \Z/p\Z$, the number of solutions $m\bmod p$ of $f_a(m)=c\bmod p$ is at most $d$, and since $p>N$, this certainly applies to those $m\leq N$ which solve $f_a(m) = c$ with $c=f_a(n)$. 

 
Moreover, if $p>N$, the maximal $k$ so that $p^k\mid f_0(n)-a$ for some $n\leq N$ is, because we assume $f_a(n)\neq 0$,  
\[
\ll \frac{ \log (N^d+|a|)}{\log p}   =O_d(1)
\]
 because we assume that $|a|\leq T$ with $\log T\ll \log N$. 

Therefore
\[
\alpha_p(N) = \sum_{k\geq 1}\#\{n\leq N: f(n)=0\bmod p^k\}  \leq \sum_{1\leq k\ll O_d(1)} d =O_d(1)
\]
\end{proof}

\subsection{A preliminary bound on $\Delta_N(a)$}
\begin{lem}\label{prelim bound for Delta}
If $a$ is such that $f_0(x)-a$ has no rational zeros, and $\log  |a|\ll \log N$, then
\begin{equation}\label{second bound for Delta} 
\Delta_N(a)\ll \sum_{\substack{1\leq m<n\leq N\\ G(m,n)\neq 0}} \sum_{\substack{N<p\ll N^{d}\\ p\mid f_0(m)-a\\ p\mid G(m,n)}} \log p
+O(\log N) 
\end{equation}
\end{lem}
\begin{proof}
We have $\alpha_p(N)\neq \beta_p(N)$ if and only if there are two distinct integers $m,n\leq N$ so that $p\mid f_a(m)$ and $p\mid f_a(n)$.  
Using Lemma~\ref{lem: bound on alpha}, we see that $\alpha_p(N)-\beta_p(N) = O_d(1)$ for $p>N$, and hence applying a union bound we obtain, if $a$ is such that $f_a(x)$ has no rational zeros, 
\[
\Delta_N(a) \ll_d \sum_{1\leq m<n\leq N} \sum_{\substack{N<p\ll N^d\\ p\mid f_0(m)-a\\p\mid f_0(n)-a}} \log p
\]
Note that  if $p\mid f_a(m)$ and $p\mid f_a(n)$ then $p\mid f_a(m)-f_a(n) = (m-n)G(m,n)$ and so since 
$p\nmid m-n$ (because $1\leq n-m\leq N-1<p$), we  must have $p\mid G(m,n)$. 
Thus 
\begin{equation}\label{first bound for Delta} 
\Delta_N(a)\ll \sum_{1\leq m<n\leq N} \sum_{\substack{N<p\ll N^{d}\\ p\mid f_0(m)-a\\ p\mid G(m,n)}} \log p
\end{equation}


We break off the terms corresponding to $G(m,n)=0$. According to Lemma~\ref{lem:G(m,n)}, the condition $G(m,n)=0$ forces $m,n\leq C_1$ to be bounded. Hence the contribution of such pairs to \eqref{first bound for Delta} is bounded by 
\[
\ll \sum_{m,n\leq C_1} \sum_{\substack{N<p\ll N^{d}\\ p\mid f_0(m)-a  }} \log p \ll \log N \max_{m\leq C_1}\#\{p>N:p\mid a-f_0(m)\} 
\]
Note that $0<|f_0(m)-a|\ll |a|+1$ if $m\leq C_1$ (we assume that $a$ is such that  $f_0(x)-a$ has no rational zeros, hence $f_0(m)-a\neq 0$, and hence the number of primes $p>N$ dividing $f_0(m)-a$ is at most $\ll \log |a|/\log N$. Hence the contribution of pairs $m<n$ with $G(m,n)=0$ to \eqref{first bound for Delta} is at most $\ll \log |a| $. Thus
\begin{equation*}   
\Delta_N(a)\ll \sum_{\substack{1\leq m<n\leq N\\ G(m,n)\neq 0}} \sum_{\substack{N<p\ll N^{d}\\ p\mid f(m)\\ p\mid G(m,n)}} \log p
+O\Big( \log |a|   \Big)
\end{equation*}
Finally, the assumption $\log |a|\ll \log N$ gives \eqref{second bound for Delta}.
\end{proof}

\subsection{Proof of Proposition~\ref{prop: mean of DeltaN}}

Now to average over $|a|\leq T$ (such that $f_0(x)-a$ is irreducible). 
Using \eqref{second bound for Delta}, noting that  $\log |a| \ll \log T\ll \log N$   gives 
\[
\ave{\Delta_N}  \ll 
\sum_{\substack{1\leq m<n\leq N\\G(m,n)\neq 0}} \sum_{\substack{ N<p\ll N^{d} \\  p\mid G(m,n) }} \log p \; \frac 1T \#\{|a|\leq T: p\mid a-f_0(m)\}
+O (\log N )
\]

Given $1\leq m<N$, and $N<p\ll N^{d}$, the number of $|a|\leq T$ with $a=f_0(m)\bmod p$ is $\ll T/p+1 $.  Hence 
\begin{equation*}
\begin{split}
\ave{\Delta_N} & \ll  \sum_{\substack{1\leq m<n\leq N\\ G(m,n)\neq 0}} \sum_{\substack{ N<p\ll  N^{d} \\  p\mid G(m,n)}} \frac{\log p}{p} 
+\frac 1T \sum_{\substack{1\leq m<n\leq N\\ G(m,n)\neq 0}}  \sum_{\substack{N<p\ll  N^{d}\\ p\mid G(m,n)}} \log p +O(\log N)
\\
& =:I+II +O(\log N)
\end{split}
\end{equation*}

To treat the sum $II$, we  note if $m,n\leq N$, then $|G(m,n)|\leq C(f_0)N^{d-1}$ and so  there are at most $d-1$ distinct primes $p>N$ which divide $G(m,n)$ (which we assume is non-zero), 
and for these $\log p\ll \log N$. Therefore 
\[
II\ll \frac{\log N}{T} \sum_{1\leq m<n\leq N}(d-1)\ll \frac{N^{2}\log N}{T}
\] 
which is $O(N)$ if $T>N\log N$. 

To treat the sum $I$, we separate the prime sum into primes with $N<p\leq N\log N$ and the remaining large primes $N\log N<p\ll N^{d-1}$ to get
\[
  I \ll\sum_{\substack{1\leq m<n\leq N\\G(m,n)\neq 0}} \sum_{\substack{ N<p<N\log N \\  p\mid G(m,n)}} \frac{\log p}{p} 
  +  \sum_{\substack{1\leq m<n\leq N\\G(m,n)\neq 0}} \sum_{\substack{ N\log N<p\ll N^{d} \\  p\mid G(m,n)}} \frac{\log p}{p} 
\]

We treat the sum over small primes by switching the order of summation
\begin{multline*}
\sum_{\substack{1\leq m<n\leq N\\G(m,n)\neq 0}}\sum_{\substack{ N<p<N\log N \\  p\mid G(m,n)}} \frac{\log p}{p}  
\\
\leq  
\sum_{ N<p<N\log N }  \frac{\log p}{p} \#\{1\leq m<n\leq N: G(m,n)=0\bmod p\}
\end{multline*}
 Now given $m$, the congruence $G(m,n)=0\bmod p$ (if solvable)  determines $n \bmod p$ up to $d-1$ possibilities, since $G(m,n)$ is a monic polynomial of degree $d-1$ in $n$, 
 and since $n\leq N<p$ it means that $n$ is determined as an integer up to $d-1$ possibilities. Hence 
 \[
 \#\{1\leq m<n\leq N: G(m,n)=0\bmod p\}\leq (d-1)N
 \]
 and the sum over small primes is bounded by 
 \begin{multline*}
\ll \sum_{ N<p<N\log N }   \frac{\log p}{p} N 
\\
=N\Big\{ \Big( \log( N\log N) +O(1) \Big) -\Big(\log N+O(1)\Big)\Big\} \sim N\log \log N 
 \end{multline*}
 on using Mertens' theorem. 
 
The sum over large primes is treated by using $\log p/p \ll 1/N$ for $p>N\log N$, 
giving 
\[
 \sum_{\substack{1\leq m<n\leq N\\G(m,n)\neq 0}} \sum_{\substack{ N\log N<p \ll N^{d} \\  p\mid G(m,n)}} \frac{\log p}{p}  \ll \frac 1{N }
 \sum_{\substack{1\leq m<n\leq N\\G(m,n)\neq 0}} \#\{p>N\log N: p\mid G(m,n)\}
\]
Now given $1\leq m<n\leq N$ with $G(m,n)\neq 0$, there are at most $d-1$ primes $p>N\log N$ dividing $G(m,n)\ll N^{d-1}$,  
so that the contribution of large primes is bounded by 
\[
\ll \frac 1N\sum_{1\leq m<n\leq N} (d-1) \ll N
\]
 This gives $I\ll N\log \log N$, and  hence 
 \[
 \ave{\Delta_N} \ll N\log\log N
 \]
 as claimed. \qed

 \section{Almost sure behaviour of $C_N$}

\subsection{}
 Let $f\in \Z[x]$ be an irreducible polynomial,  and let $\rho_f(p)$ be the number of distinct roots of the polynomial $f$ modulo a prime $p$.  
It is well known that for fixed $f$, the mean value of $\rho_f(p)$ over all primes is $1$ \cite{Nagell 1921}:
 \[
 \frac 1{\pi(x)}\sum_{p\leq x} \rho_f(p) =1+o_f(1) .
 \]
We write
 \[
 \rho_f(p) 
 =1+\sigma_f(p)
 \]
 where $\sigma_f(p)$ is a fluctuating quantity, having mean zero. 

 Now fix $f_0(x)=x^d+c_{d-1}x^{d-1}+\dots +c_1x\in \Z[x]$, a monic polynomial of degree $d$, and for $a\in \Z$ set
 \[
 f_a(x) = f_0(x)-a
 \]
 Write $\rho(a;p) = \rho_{f_a}(p)$, $\sigma(a;p)=\sigma_{f_a}(p)$. Note that $0\leq \rho(a;p)\leq d$. 
 
   We write  
\begin{equation*}
 C_N(a):=  \sum_{\substack{ p\leq N \\p\nmid \disc(f)}}\frac{\log p}{p-1 }\rho (a;p) 
 = \sum_{ p\leq N }\frac{\log p}{p} -E_N(a) + D_N(a) +O(1)
\end{equation*}
where
\[
D_N(a):=\sum_{\substack{ p\leq N \\p\nmid \disc(f_a)}}\frac{\log p}{p }\sigma(a;p)
\]
and
\[
E_N(a):=\sum_{\substack{p\leq N\\p\mid D(a)}}\frac{\log p}{p }
\]

By Mertens' theorem 
\[
\sum_{p\leq N}\frac{\log p}{p }  = \log N +O(1)
\]

The contribution $E_N(a)$ of primes dividing the discriminant $D(a) = \disc(f_0(x)-a)$ can be bounded individually, for $|a|\leq T\ll N^d$, using Lemma~\ref{lem divisor sum} (assuming $D(a)\neq 0$)
\[
E_N(a) =\sum_{\substack{p\leq N\\p\mid D(a)}}\frac{\log p}{p }\leq \sum_{ p\mid D(a)}\frac{\log p}{p } \ll \log\log |D(a)|
\]
Since $D(a)$ is a polynomial of degree $d-1$ in $a$, and $|a|\leq T\ll N^d$, we find
\[
\sum_{\substack{p\leq N\\p\mid D(a)}}\frac{\log p}{p }\ll  \log\log N 
\]
which is negligible relative to the main term.  
Hence
\[
C_N(a) = \log N  + D_N(a) +O(\log\log N)
\]
In the following part, we will establish the following upper bound on the second moment of $D_N(a)$:
\begin{prop}\label{mean square D}
For  $T\ge N\log N$, the second moment of $D_N(a)$ satisfies:
$$\ave{|D_N|^2}\ll 1$$
\end{prop}

Using the triangle inequality and Cauchy-Schwartz, we obtain
\begin{prop}
$$\ave{|C_N - \log N|^2}\ll (\log\log N)^2$$
\end{prop}

As a consequence, we deduce our main objective for this section:
\begin{prop}
For almost all $|a|\leq T$ (with $N\log N\leq  T\ll N^{d-1}$)
\[
C_N(a)=\log N+O(\log\log N)
\]
\end{prop}


\subsection{Proof of Proposition~\ref{mean square D}}

 \begin{proof}
 Expanding, we have
 \[
 \ave{(D_N)^2}   = \sum_{p\leq N} \sum_{q\leq N}\frac{\log p 
\log q}{pq} \ave{\sigma(a;p)\sigma(a;q)} 
 \]

The diagonal contribution $p=q$ gives
\[
\sum_{p\leq N} \frac{(\log p )^2}{p^2} \ave{\sigma(a;p)^2}  
\]
Now note that 
\[
-1\leq \sigma(a;p) \leq d-1
\]
is uniformly bounded. This is because the polynomial $f_0(x)-a$ is monic of degree $d$, hence has at most $d$ zeros modulo $p$, so that $0\leq \rho(a;p)\leq d$ and so $-1\leq \sigma(a;p)\leq d-1$. Thus we obtain a bound for the diagonal sum 
\begin{equation*}
\sum_{p\leq N} \frac{(\log p )^2}{p^2} \ave{\sigma(a;p)^2}   \ll 
\sum_{p\leq N} \frac{(\log p )^2}{p^2} 
\ll 1
 \end{equation*}

For the off-diagonal terms, we use 
 \begin{lem}\label{lem:cov rho(p) rho(q)}
 For distinct primes $p\neq q$, 
 \[
  |\ave{\sigma(\bullet;p)\sigma(\bullet;q)}|\ll   \frac  {  \sqrt{pq}  \log(pq) }{T} +\frac 1{\sqrt{T}}
  \]
  \end{lem}
  Therefore, given Lemma~\ref{lem:cov rho(p) rho(q)}, we obtain
\begin{equation*}
\begin{split}
\sum_{ p\neq q\leq N} \frac{\log p\log q}{pq} |\ave{\sigma(a;p)\sigma(a;q)}|  
&\ll
 \sum_{p\neq q\leq N}  \frac{\log p\log q}{pq} \Big( \frac  {  \sqrt{pq}  \log(pq) }{T}+\frac 1{\sqrt{T}}\Big)
 \\
 &\ll \frac{\log N}{T}(\sum_{p\leq N} \frac{\log p}{\sqrt{p}})^2 +\frac 1{\sqrt{T}}(\sum_{p\leq N} \frac{\log p}{p})^2
 \\
 & \ll \frac{N\log N}{T} +\frac{(\log N)^2}{\sqrt{T}}
\end{split}
\end{equation*}
 which is $O(1)$ if $T\geq N\log N$, proving Proposition~\ref{mean square D}. 
 \end{proof}

  \subsection{Proof of Lemma~\ref{lem:cov rho(p) rho(q)}}
For the argument, it will be important to have $a$ run over an interval. So we first remove the restriction in the averaging, that $f_0(x)-a$ is irreducible. Since $-1\leq \sigma(a;p)\leq d-1$, this introduces an error bounded by 
  \[
\ll  \frac 1T \sum_{\substack{ |a|\leq T\\f_0(x)-a \;{\rm reducible}}} (d-1)^2 \ll \frac 1T\#\{|a|\leq T: f_0(x)-a \;{\rm reducible}\} \ll \frac 1{\sqrt{T}}
  \]
  and so 
  \[
  \ave{\sigma(a;p)\sigma(a;q)}  = \frac 1{2T+O(\sqrt{T})}\sum_{|a|\leq T} \sigma(a;p)\sigma(a;q) +O(\frac 1{\sqrt{T}})
  \]

 We express $\rho(a;p)$ as an exponential sum:
 \[
 \rho(a;p) = \#\{x\bmod p: f_0(x)-a=0\bmod p\}=\sum_{x\bmod p} \frac 1p\sum_{t\bmod p} e(\frac{t(f_0(x)-a)}{p})
 \]
 The term $t=0$ contributes the main term of $1$, and we obtain the following expression for $\sigma(a;p)=\rho(a;p)-1$: 
 \begin{equation}\label{expression for sigma}
 \sigma(a;p) =  \frac 1p\sum_{t\neq 0\bmod p} e(-\frac{at}{p})\sum_{x\bmod p} e(\frac{tf_0(x)}{p})
 \end{equation}
 where $e(z):=e^{2\pi iz}$. Set 
 \[
 \mathcal S_{f_0}(b,n) :=\sum_{x\bmod n} e(\frac{bf_0(x)}{n}) 
 \]
 
 Using \eqref{expression for sigma}, we have on switching orders of summation
   \begin{multline*}
 \frac 1{2T+O(\sqrt{T})} \sum_{|a|\leq T}\sigma(a;p)\sigma(a;q)
 \\
 = \frac 1{2T+O(\sqrt{T})}\frac 1{pq} \sum_{\substack{0\neq t\bmod p\\ 0\neq s\bmod q}}
 \sum_{|a|\leq T}e(-a(\frac tp+\frac sq)) \mathcal S_{f_0}(t,p)\mathcal S_{f_0}(s,q)
 \end{multline*}
 
  Weil's bound \cite{weilPNAS, Schmidt} shows that there is a constant $c(d)>0$, so that all primes 
 $p$ and $b$ coprime to $p$
  \begin{equation}\label{Weil bd for S}
 | \mathcal S_{f_0}(b,p)|\leq c(d) \sqrt{p}
 \end{equation} 
In fact for any $f_0\in \Z[x]$ with $f_0(x)$ primitive of degree $d$, if $p> d$ then $|S_{f_0}(b,p)|\leq (d-1)\sqrt{p}$.

 Hence we find
\begin{equation*}
\begin{split}
  |\ave{\sigma(\bullet;p)\sigma(\bullet;q)}|&\ll_d \frac 1{T \sqrt{pq}} \sum_{\substack{0\neq t\bmod p\\ 0\neq s\bmod q}}
 |\sum_{|a|\leq T}e(-a(\frac tp+\frac sq)) | +O(\frac 1{\sqrt{T}})
 \\
 & = \frac 1{T \sqrt{pq}} \sum_{\substack{m\bmod pq\\ \gcd(m,pq)=1}} |\sum_{|a|\leq T}e(-\frac{am}{pq} ) |+O(\frac 1{\sqrt{T}})
 \end{split}
 \end{equation*}
 where we have used that if $p\neq q$ are distinct primes, then as $t$ and $s$ vary over all invertible residues modulo $p$ (resp., modulo $q$), $tq+sp\bmod pq$ covers all invertible residues modulo $pq$ exactly once.

 We   sum the geometric progression
 \[
 | \sum_{|a|\leq T}e(-\frac{am}{pq} ) | \ll \min\Big(T,  ||\frac {m}{pq}||^{-1} \Big)
 \]
 where $||\alpha|| = \dist(\alpha,\Z)$.  We may take $1\leq m<pq/2$ and then the bound is $\ll pq/m$.  
 This will give 
 \[
 \begin{split}
  |\ave{\sigma(\bullet;p)\sigma(\bullet;q)}| &\ll  \frac 1{T \sqrt{pq}} \sum_{\substack{1\leq m \leq   pq/2\\ \gcd(m,pq)=1}} \frac{pq}{m} +O(\frac 1{\sqrt{T}})
  \\
 & \ll \frac{\sqrt{pq}\log(pq)}{T}+O(\frac 1{\sqrt{T}})
\end{split}
  \]
 proving Lemma~\ref{lem:cov rho(p) rho(q)}. \qed

 \end{document}